\documentclass[12pt, a4paper]{amsart}
\usepackage{thmtools}

\usepackage{marginnote}
\usepackage{listings}
\usepackage{amssymb,amsmath,amsthm}
\usepackage{tikz-cd}
\usetikzlibrary{arrows}
\usepackage{graphicx}
\usepackage[hidelinks]{hyperref}
\usepackage{mathtools}
\usepackage[titletoc,title]{appendix}
\usepackage{graphicx}
\usepackage{caption}
\usepackage{subcaption}
\usepackage{blindtext}
\usepackage[titletoc,title]{appendix}
\usepackage[english]{babel}
\usepackage{mathrsfs}
\usepackage{cite}
\usepackage{mathabx}
\usepackage{microtype}
\usepackage[capitalize]{cleveref}

\newcommand{\F}{\mathbb{F}}

\newtheorem{theorem}{Theorem}[section]
\newtheorem{lemma}[theorem]{Lemma}
\newtheorem{corollary}[theorem]{Corollary}
\newtheorem{proposition}[theorem]{Proposition}
\theoremstyle{definition}

\newtheorem{remark}[theorem]{Remark}

\newtheorem{question}[theorem]{Question}

\newtheorem*{theorem*}{Theorem}
\newtheorem*{conjecture*}{Conjecture}

\newcommand{\Z}{\mathbb{Z}}

\makeatletter
\def\Ddots{\mathinner{\mkern1mu\raise\p@
\vbox{\kern7\p@\hbox{.}}\mkern2mu
\raise4\p@\hbox{.}\mkern2mu\raise7\p@\hbox{.}\mkern1mu}}

\makeatother

\newcommand{\normal}[1]{\langle\!\langle #1 \rangle\!\rangle}

\DeclareSymbolFontAlphabet{\amsmathbb}{AMSb}

\DeclareMathOperator{\rk}{rk}

\DeclarePairedDelimiter\abs{\lvert}{\rvert}

\makeatletter
\let\oldabs\abs
\def\abs{\@ifstar{\oldabs}{\oldabs*}}

\newcounter{cases}
\newcounter{subcases}[cases]

\newcounter{marcocomments}

\newcounter{bobcomments}

\begin{document}

\title[Complete rewriting systems for graphs of free groups]{
Finite complete rewriting systems for graphs of free groups with 
applications to free-by-cyclic, one-relator, and three-manifold groups}
\author[R.~D.~Gray]{Robert D. Gray}
\address{School of Engineering, Mathematics and Physics, University of East Anglia, Norwich NR4 7TJ, UK}
\email{Robert.D.Gray@uea.ac.uk}

\author[M. Linton]{Marco Linton}
\address{Instituto de Ciencias Matem\'{a}ticas, CSIC-UAM-UC3M-UCM, Madrid, Spain}
\email{marco.linton@icmat.es}

\thanks{This research was supported by the EPSRC Fellowship grant EP/V032003/1 ‘Algorithmic, topological and geometric aspects of infinite groups, monoids and inverse semigroups’. The second author was supported by the grant 202450E223 (Impulso de líneas científicas estratégicas de ICMAT)}


\maketitle

\begin{abstract}
We prove that any finite graph of finitely generated free groups admits a finite complete rewriting system after possibly taking a free product with a free group of rank two. As a corollary of our main result, we obtain that any HNN-extension of a finitely generated free group over a finitely generated subgroup admits a finite complete rewriting system. The proof of this result involves establishing a combination theorem for finite complete rewriting systems of graphs of groups. We then use this result, together with other tools, to give partial solutions to several fundamental open problems about finite complete rewriting systems for hyperbolic, one-relator, fully residually free, and three-manifold groups. 
In particular we prove that if $G = \langle \F, t \mid t^{-1}ft = \psi(f), \, \forall f\in \F\rangle$ is the mapping torus of an injective endomorphism $\psi\colon \F\to \F$ of a free group $\F$ (of possibly infinite rank) then every finitely generated subgroup of $G$ admits a finite complete rewriting system. It follows that any finitely generated virtually free-by-cyclic group, and any finitely generated subgroup of such a group, admits a finite complete rewriting system. 
We apply this to show that every finitely generated subgroup of a locally quasi-convex hyperbolic and virtually compact special group admits a finite complete rewriting system. This includes all one-relator groups with torsion (and all their finitely generated subgroups) and all hyperbolic fully residually free groups. Moreover, we show there is an algorithm that computes a finite complete rewriting system for any such group, given a presentation for its containing group and a finite list of generators for the subgroup. We also prove that for every compact three-manifold $M$, the group $\pi_1(M) \ast \Z$ admits a finite complete rewriting system. Furthermore, we show that the fundamental group of any compact three-manifold is autostackable and thus has a rational cross section and admits a bounded regular convergent prefix-rewriting system.  
\end{abstract}

\section{Introduction and summary of results}

A presentation for a group may be viewed as a rewriting system which in some cases can be used to give solutions to fundamental algorithmic problems for the group. For example, if the rewriting system is finite and complete (also called convergent, which means that it is both confluent and terminating) then this gives an algorithm for computing a regular set of normal forms for the elements of the group and thus gives a solution to the word problem. Once a finite complete rewriting system for a group is known, the algorithm for computing normal forms using the system is then easy to implement on a computer. Rewriting methods in group theory have a rich history going back to classical work of Dehn and Thue in the early 20th-century. Numerous classical algorithms in group theory may be viewed as string rewriting processes e.g. the Nielsen method in free groups, Dehn's algorithm in hyperbolic groups, and Tits' algorithm in Coxeter groups. As well as providing a nice solution to the word problem, and a regular set of normal forms, complete rewriting systems for groups have also been applied to prove other results e.g. in \cite{Grigorchu1997} they use a particular finite complete rewriting systems for surface groups to explicitly compute the complete growth functions of surface groups. Rewriting systems also have important applications in rings and algebras e.g. the theory of Gr\"{o}bner--Shirshov bases and Buchberger’s algorithm for computing normal forms and solving ideal-membership problems; see \cite{AdamsGrobner}. 

The question of which groups admit finite complete presentations has received a lot of attention in the literature. Finite complete rewriting systems have been constructed for Baumslag--Solitar groups \cite[p429]{HoltBook}, surface groups \cite{LeChenadec1986, hermiller1994rewriting, Grigorchu1997}, torus knot groups \cite{Dekov1997}, alternating knot groups \cite{Chouraqui2009}, virtually abelian groups, polycyclic groups, and constructible solvable groups \cite{Groves1993}, 
Coxeter groups of large type \cite{hermiller1994rewriting}, 
Artin groups of finite type \cite{Hermiller1999Second}, 
right-angled Artin groups  \cite{Hermiller1995}, iterated Baumslag--Solitar (Gersten) groups 
\cite{hermiller2001measuring}, and the Greendlinger group\cite{Otto1984}. 

However, many fundamental questions about which groups admit finite complete remain open. Some key open questions in this area include the following: 
\begin{enumerate} 
\item
Do all hyperbolic groups have finite complete rewriting systems? See \cite[Problem~2.6]{Diekert2010}.
\item 
Does every one-relator group have a finite complete rewriting system? See \cite{Otto1997}, \cite[Problem 2.9]{Diekert2010} and \cite[Section~1.8.3.]{linton2025theory}.
\item
Does every finitely generated fully residually free group have  a finite complete rewriting system? See 
\cite[Problem~2.7]{Diekert2010}.
\item
Does every finitely generated three-manifold group have a finite complete rewriting system? See \cite{Hermiller1999}.
\end{enumerate}
In this paper we shall add the class of fundamental groups of finite graphs of finitely generated free groups to the list of groups admitting finite complete rewriting systems, with a caveat. Then we shall show how our result may be applied to give partial solutions to each of the four open problems listed above.

There will be a slight technicality which will arise in the proof of our main theorem which means we actually establish \emph{stably} finite complete rewriting systems: that is, $G$ stably admits a finite complete rewriting system if $G*F_k$ admits a finite complete rewriting system for some $k\geqslant 0$. More precisely, we shall prove:

\begin{restatable}{theorem}{MainTheoremA}
\label{thm:main_free}
If $G$ splits as a finite graph of finitely generated free groups, then $G$ stably admits a finite complete rewriting system. More precisely, if $\Gamma$ is the underlying graph of the graph of groups, then:
\begin{enumerate}
\item $G*F_k$ admits a finite complete rewriting system, where $k = \min\{2, |V(\Gamma)| - 1\}$.
\item $G*\Z$ admits a finite complete rewriting system if $G$ contains a proper finite index subgroup.
\end{enumerate} 
\end{restatable}

As immediate corollaries, we obtain:

\begin{corollary}\label{cor:FreeProdAmalg}
If $G = A*_CB$ with $A, B, C$ finitely generated free, then $G*\Z$ admits a finite complete rewriting system.
\end{corollary}

\begin{corollary}
\label{cor:HNN}
If $G = A*_C$ with $A, C$ finitely generated free, then $G$ admits a finite complete rewriting system.
\end{corollary}

As a slightly less immediate corollary, using the Dunwoody resolution (see \cite[Lemma 1]{BF91}) we have: 

\begin{corollary}
Every finitely presented subgroup of a graph of free groups stably admits a finite complete rewriting system.
\end{corollary}

We shall prove Theorem~\ref{thm:main_free} by first introducing (see Section~\ref{sec:combining}) the notion of a finite complete rewriting system compatible with a collection of subgroups, and then proving a combination result for graphs of groups in Theorem~\ref{thm:combination}. Associated with a graph of groups is a group $F(\mathcal{G}, \Gamma)$ which naturally arose in Serre's development of the theory of graphs of groups \cite{Serre} and is a natural homomorphic pre-image of the fundamental group $\pi_1(\mathcal{G})$. Our combination result Theorem~\ref{thm:combination} gives sufficient conditions for constructing a finite complete rewriting system for $F(\mathcal{G}, \Gamma)$ from complete rewriting systems of the vertex groups. The proof of Theorem~\ref{thm:main_free} makes use of this combination result together with \cref{thm:compatible_free} which states that the standard finite complete rewriting system of a finitely generated free group is compatible with any collection of finitely generated free groups.

Combining \cref{thm:main_free} with a result of Feighn--Handel \cite{Feighn1999} 
and a result of Groves--Smith \cite{Groves1993}
we obtain a more concrete result:

\begin{restatable}{theorem}{MainTheoremB}
\label{rewriting_main}
If $\F$ is a (not necessarily finitely generated) free group, $\psi\colon \F\to \F$ is a monomorphism and 
\[
G = M(\psi) := \langle \F, t \mid t^{-1}ft = \psi(f), \, \forall f\in \F\rangle
\]
is a mapping torus of a free group, then any finitely generated subgroup of $G$ admits a finite complete rewriting system. 

In particular, any finitely generated virtually free-by-cyclic group, and any finitely generated subgroup of such a group, admits a finite complete rewriting system. 
\end{restatable}

It appears that finite complete rewriting systems were not known to exist even for mapping tori of finitely generated free groups. However, it is important to stress that the fact that any \{finitely generated free\}-by-cyclic group (the case when $\F$ is finitely generated and $\psi$ is an isomorphism) admits a finite complete rewriting system is already known and follows from \cite[p286]{Groves1993}. Theorem~\ref{rewriting_main} says much more than this since it applies to any \{finitely generated\} free-by-cyclic group and thus applies to any finitely generated subgroup of a  \{finitely generated free\}-by-cyclic group $F_n \rtimes \Z$. To obtain applications like Corollary~\ref{cor:PrimRank} below for one-relator groups with torsion (and the other applications in this paper), one needs Theorem~\ref{rewriting_main} for general \{finitely generated\} free-by-cyclic groups since one-relator groups with torsion, for example, are not virtually \{finitely generated free\}-by-cyclic. We note that it has recently been proved \cite{Li26} that in fact the \{finitely generated\} free-by-cyclic groups are precisely the finitely generated subgroups of the \{finitely generated free\}-by-cyclic groups.

Theorem~\ref{thm:main_free} and Theorem~\ref{rewriting_main} can be applied to give partial solutions to each of the four fundamental open problems listed above. This is possible due to the fact that quite a number of groups are now known to be virtually free-by-cyclic.  After proving Theorem~\ref{thm:main_free}, then in the second half of the paper, we shall give some of these applications. Some of the applications are immediate consequences of Theorem~\ref{rewriting_main} while others, like our applications to three-manifold groups, require a little more work to establish. In particular the proofs of our results on three-manifold groups will also involve further applications of our combination result Theorem~\ref{thm:combination}. We shall now give a summary of the main applications of Theorem~\ref{thm:main_free} and Theorem~\ref{rewriting_main} that we will establish in this paper. 

It was proved in \cite[Theorem A.2]{fisher2026coherent}
that if $G$ is a virtually compact special and locally quasi-convex hyperbolic group, then $G$ is virtually free-by-cyclic. Using that result we will apply Theorem~\ref{rewriting_main} to prove the following.  

\begin{theorem}\label{thm:locallyQuasConv}
If $G$ is a locally quasi-convex hyperbolic and virtually compact special group, then finitely generated subgroups of $G$ admit finite complete rewriting systems. Moreover, there is an algorithm which, on input a finite presentation for $G$ and a finite subset $S\subset G$, computes a finite complete rewriting system for the subgroup $H = \langle S\rangle\leqslant G$.
\end{theorem}

Then applying \cite[Theorem 1.4]{linton2025geometry} we obtain the following which gives a partial solution to the open question of whether all one-relator groups admit finite complete rewriting systems. We denote by $\pi(w)$ the primitivity rank of the relator word $w \in F$ which is defined as follows:    
\[
\pi(w) = \mathrm{min}\{
\mathrm{rk}(H) \mid 
w \in H \leq F, w \; \mbox{not primitive in} \; H \} \in \mathbb{N} \cup \{ \infty \}. 
\]

\begin{corollary}\label{cor:PrimRank}
If $G = F/\normal{w}$ is a finitely generated one-relator group with $\pi(w)\neq 2$ (e.g. if $G$ is a one-relator group with torsion) then all finitely generated subgroups of $G$ admit finite complete rewriting systems. Moreover, there is an algorithm which, for any finite subset $S\subset G$, computes a finite complete rewriting system for the subgroup $H = \langle S\rangle\leqslant G$.
\end{corollary}

Theorem~\ref{thm:locallyQuasConv} also applies to some other one-relator groups not covered by Corollary~\ref{cor:PrimRank}. 
For example, by\cite[Corollary 1.5]{KL24b}  it may be applied to prove that small cancellation one-relator groups, and all their finitely generated subgroups, all admit finite complete rewriting systems. Since generic groups in the few relator model (and thus, generic one-relator groups) are small cancellation, this implies that a generic one-relator group (and all its finitely generated subgroups) admits a finite complete rewriting system.

In addition to the above classes of one-relator groups, it is proved by Fisher in \cite{fisher2026cohomological}  that any virtually RFRS one-relator group is virtually free-by-cyclic and so Theorem~\ref{rewriting_main} applies to show all of those one-relator groups admit finite complete rewriting systems. Wise \cite[Conjecture 17.8]{WiseWhatsNext} has conjectured that every hyperbolic one-relator group is virtually free‑by‑cyclic and so if that conjecture holds it would clearly follow from Theorem~\ref{rewriting_main} that every hyperbolic one-relator group admits a finite complete rewriting system.
Also, Theorem~\ref{rewriting_main} applies to one-relator groups that are (virtually) mapping tori of free groups e.g. Baumslag--Solitar groups and other examples such as the non-linear residually finite one-relator group of Drutu and Sapir \cite{DrutuSapir2005}.  
 
As mentioned above, it is an open problem as to whether all one-relator groups admit finite complete rewriting systems. It is even possible that every finitely generated subgroup of a one-relator groups admits a finite complete rewriting system. It follows from the Anick--Groves--Squier Theorem \cite{AGS} that if a group admits a finite complete rewriting system then it must be of type $\mathrm{FP}_\infty$. It may be shown\footnote{Since one-relator groups with torsion are coherent, every finitely generated subgroup of a one-relator group with torsion is finitely presented and thus is hyperbolic by \cite{Gersten} and hence is of type $\mathrm{FP}_\infty$ using the Rips complex. If $G$ is a torsion-free one-relator group then $G$ has cohomological dimension two and so all its subgroups have cohomological dimension two. Since $G$ is coherent \cite{JZL2025}, every finitely generated subgroup has type $\mathrm{FP}_2$ and cohomological dimension two and so has type $\mathrm{FP}_{\infty}$.} that every finitely generated subgroup of a one-relator group is of type $\mathrm{FP}_\infty$, so the standard obstruction to having a finite complete rewriting system is not possible here.  

Recall that a limit group is a finitely generated fully residually free group. Regarding the third of the fundamental open problems we list above, which asks about the existence of complete rewriting systems for limit groups \cite[Problem 2.7]{Diekert2010}, we may apply Theorem~\ref{rewriting_main} to prove the following:  

\begin{corollary}\label{theorem:LimitGroups}
Limit groups with cohomological dimension $2$   (equivalently, limit groups without $\Z^3$-subgroups) admit finite complete rewriting systems. In particular every  hyperbolic limit group admits a finite complete rewriting system.
\end{corollary}

This corollary follows from Theorem~\ref{rewriting_main} since it was proved by Hagen--Wise \cite[Corollary B]{HW2010} that every hyperbolic limit group (and more generally any limit group not embedding $\mathbb{Z}^3$) is virtually free-by-cyclic. 

The final problem to which we apply Theorem~\ref{rewriting_main} is the open question of whether all thee-manifold groups admit finite complete rewriting systems. In \cite{Hermiller1999}, Hermiller--Shapiro show that all closed geometric three-manifolds admit finite complete rewriting systems. Their proof was conditional on the virtual fibring conjecture for the hyperbolic case, which has been solved since their paper was published. At the end of \cite{Hermiller1999} they ask whether every fundamental group of a closed three-manifold admits a finite complete rewriting system. In Section~\ref{sec:3man} we shall prove the following result, almost answering their question.  

\begin{restatable}{theorem}{MainTheoremC}
\label{thm:3-manifold:intro}
If $M$ is a compact three-manifold, then $\pi_1(M) \ast \Z$ admits a finite complete rewriting system.  
\end{restatable}

Note that while the question in \cite{Hermiller1999} only asked about closed three-manifolds our result applies to fundamental groups of arbitrary compact three-manifolds.   

We also note that it is an open problem if a free product of groups $H \ast K$ admitting a finitely complete rewriting system implies that both $H$ and $K$ also do; see \cite{Pride_Wang_2000} which also explains how that problem relates to the open question of whether having a finite complete rewriting system is a property preserved under taking finite index subgroups of groups. One challenge when trying to prove that having a finite complete rewriting system might pass from $\pi_1(M)*\Z$ to $\pi_1(M)$ is the fact that this property  is sensitive to the choice of finite generating set for the group; indeed, Jantzen showed in \cite{jantzen1985note} that the Klein bottle group $\langle a, b \mid a^2b^2\rangle$ does not admit a finite complete rewriting system on the generating set $\{a, b\}$, although it does over another generating set.

We do not know whether all three-manifold groups admit finite complete rewriting systems. Nevertheless, our results do show the following.  
\begin{theorem}\label{thm:GoodPieces}
If $M$ is a compact orientable three-manifold whose prime decomposition pieces do not contain any closed graph manifolds which do not virtually fibre, then $\pi_1(M)$ admits a finite complete rewriting system. Furthermore any finitely generated subgroup of $\pi_1(M)$ admits a finite complete rewriting system. 
\end{theorem}
The above does not cover all compact three-manifolds since there exist closed orientable graph manifolds which do not virtually fiber \cite{LYQ97}. Theorem~\ref{thm:GoodPieces} reduces the problem of whether all three-manifold groups admit complete rewriting systems to certain cases, since it implies:

\begin{corollary}\label{thm:3ManReduction} 
If every fundamental group of a closed orientable graph manifold that does not virtually fibre admits a finite complete rewriting system then every fundamental group of a compact three-manifold admits a finite complete rewriting system.
  \end{corollary}

This corollary follows from Theorem~\ref{thm:GoodPieces} since admitting a finite complete rewriting system is closed under free products and finite extensions. See \cref{prop:conditional} for details.

It is proved in \cite{behrstock2008quasi} that the fundamental groups of any two closed irreducible non-geometric graph-manifolds are quasi-isometric. 
However, there are non-geometric closed graph manifold groups that do admit finite complete rewriting systems (see e.g. \cite[Section~6]{Hermiller1999}). 
Combining these observations with \cite[Theorem 0.1]{papasoglu2002quasi}, 
we conclude that another consequence of Theorem~\ref{thm:GoodPieces} is the following: 

\begin{corollary} 
Every three-manifold group is quasi-isometric to a three-manifold group
that admits a finite complete rewriting system.
\end{corollary}

It is unknown whether the property of having a finite complete rewriting system is a quasi-isometry invariant. It is also unknown whether this property is preserved when passing to a finite index subgroup of a group. 

Recall that a rational cross-section for a group is a regular language of unique representatives for elements of the group. 
See \cite{Bodart2024} for recent results on groups with rational cross sections, including the first examples of finitely presented groups with solvable word problem and without rational cross-sections.
Since the property of admitting a rational cross section is passed to free factors (see e.g. \cite[page 2]{Bodart2024}) and since any group admitting a finite complete rewriting system has a rational cross section, one immediate consequence of Theorem~\ref{thm:3-manifold:intro} is the following:  
  
\begin{corollary} \label{cor:3-manifold} 
Every fundamental group of a compact three-manifold 
has a rational cross section.  \end{corollary}

In fact we can apply our results to say more than this.    
As explained in the introduction to \cite{Hermiller2018}, one of the original motivations for the definition of automatic groups was to find algorithmic and topological characteristics of the Cayley graph satisfied by three-manifold fundamental groups, to allow computations to be carried out in these groups. There are, however, three-manifold groups that are not automatic.  

Autostackable groups, which were introduced by Brittenheim--Hermiller--Holt in the paper \cite{Brittenham2014}, are a generalisation of  
both automatic groups and groups with finite complete rewriting systems. Autostackability is a topological property of the Cayley graph of a group, together with a certain language theoretic restriction on this property. 
We refer the reader to \cite{Brittenham2014} for the formal definition of this notion. 
Among other things, an autostackable structure for a finitely generated group gives a solution to the word problem, 
a rational cross section, 
a recursive algorithm for building van Kampen diagrams, and implies tame combability; see 
\cite{Brittenham2014, brittenham2015uniform}.

In \cite[Theorem 6.2.]{Hermiller2018} Brittenheim--Hermiller--Susse show that if $M$ be a compact three-manifold with incompressible toral boundary then $\pi_1(M)$ is autostackable. In particular, if $M$ is closed, then $\pi_1(M)$ is autostackable. Applying our results we will extend this result to the general case by proving the following.

\begin{theorem}\label{thm:autostackable:intro} 
The fundamental group of any compact three-manifold is autostackable.
\end{theorem}

This result gives a common framework to solve the word problem in arbitrary finitely generated three-manifold groups using finite state automata.

In a different direction, Bridson proved \cite{bridson1993combings} that the fundamental group of any compact three-manifold is asynchronously combable. It is interesting to compare Theorem~\ref{thm:autostackable:intro} and the fact that all three-manifold groups have a rational cross section with results in \cite{BG1996} showing that if one in addition insists that a combing of the fundamental group of a three-manifold satisfies the asynchronous fellow-traveller property then the combing cannot be context-free. However, their results show that every three-manifold group admits a combing that satisfies the asynchronous fellow-traveller property and is an indexed language.

Autostackability also has an alternative characterisation in terms of rewriting systems. It is shown in \cite{Brittenham2014} that a group $G$ is autostackable if and only if $G$ admits a synchronously regular bounded convergent prefix-rewriting system. Such a system consists of an infinite set of rewrite rules $(u,v)$ which are a synchronously regular language, and where only rewritings of the form $u \beta \rightarrow v \beta$ for $(u,v) \in R$ and $\beta \in A^*$ are allowed (in contrast to an ordinary rewriting system that allows     
rewritings of the form $\alpha u \beta \rightarrow \alpha v \beta$ for $(u,v) \in R$ and $\alpha, \beta \in A^*$). 
Saying that the prefix-rewriting system is \emph{bounded} means that the set $A$ is finite and there is a constant $k$ such that for every rule $u \rightarrow v$ in $R$ we can decompose the words $u = w \gamma$ and $v = w \delta$ with $w, \gamma, \delta \in A^*$ such that $|\gamma| + |\delta| \leq k$. We refer the reader to \cite[Section~1]{Brittenham2014} for the full definition of a   
synchronously regular bounded convergent prefix-rewriting system, and the proof that it is equivalent to being autostackable. Hence another consequence of the results in this paper is the following: 

\begin{corollary} 
Every fundamental group of a compact three-manifold admits a bounded regular convergent prefix-rewriting system.
\end{corollary}

\section{Combining finite complete rewriting systems}
\label{sec:combining}

\subsection{Rewriting systems}

Let $X$ be a finite set, denote by $X^*$ the free monoid generated by $X$. A \emph{rewriting system} is a pair $(X, \mathcal{R})$ where $X$ is a finite set and $\mathcal{R}$ is a set of ordered pairs $(u, v)$ from $X^*\times X^*$, called \emph{rewriting rules}. Define a relation $\to_{\mathcal{R}}$ on words such that $x\to_{\mathcal{R}} y$ if $x = x_1ux_2$ and $y = y_1vy_2$ for some $(u, v)\in \mathcal{R}$. Then define the relation $ \to_{\mathcal{R}}^*$ on words such that $x\to_{\mathcal{R}}^*y$ if there is a sequence of words $x = x_0, x_1, \ldots, x_n = y$ such that $x_i\to_{\mathcal{R}} x_{i+1}$ for all $0\leqslant i<n$. An element $x$ is \emph{irreducible} if there is no $y\neq x$ such that $x\to_{\mathcal{R}} y$. It is \emph{reducible} otherwise.

A rewriting system $(X, \mathcal{R})$ is \emph{finite} if $X$ and $\mathcal{R}$ are finite. It is \emph{terminating} if there is no infinite sequence $x_0\to_{\mathcal{R}}x_1\to_{\mathcal{R}}\ldots$. It is \emph{confluent} if for all $x, y, z$ such that $x\to_{\mathcal{R}}^*y, z$, there exists some $\bar{x}$ such that $y, z\to_{\mathcal{R}}^*\bar{x}$. It is \emph{complete} if it is both terminating and confluent. In a complete rewriting system, each word has a \emph{normal form} which is obtained by simply applying reductions until no more reductions are possible.

\subsection{Rewriting systems for groups}

A group $G$ has a \emph{finite complete rewriting system} if there is a finite complete rewriting system $(X, \mathcal{R})$ such the monoid defined by the monoid presentation  $\langle X \mid \mathcal{R} \rangle$ is isomorphic to $G$. Equivalently, this means that there is a monoid epimorphism $\pi\colon X^*\to G$ such that $\pi(x) = \pi(y)$ if and only if the normal forms of $x$ and $y$ are equal as words in $X^*$.

We shall denote by $F_k$ the free group of rank $k$.

\begin{proposition}
\label{prop:finite_index}
Let $G$ be a group and let $H\leqslant G$ be a finite index subgroup. Then:
\begin{enumerate}
\item If $H$ admits a finite complete rewriting system, then so does $G$.
\item If $H*F_n$ admits a finite complete rewriting system, then $G*F_k$ admits a finite complete rewriting system for any $k\geqslant n/[G:H]$.
\end{enumerate}
\end{proposition}

\begin{proof}
The first statement is due to Groves--Smith \cite{Groves1993}. For the second statement, consider the group $G*F_k$. Let $T$ be a left transversal for $G/H$. Then the subgroup $\left\langle H, \{F_k^t\}_{t\in T}\right\rangle = H*(*_{t\in T}F_k^t)\cong H*F_n$, where $n = [G:H]k$, has finite index in $G*F_k$. Now the result follows by combining \cref{lem:join} with the first statement.
\end{proof}

\subsection{Relative and compatible rewriting systems}

We say that a group $G$ admits a \emph{finite complete rewriting system (left) relative to a collection of subgroups $\{G_{\alpha}\}_{\alpha\in A}$} if it admits a finite complete rewriting system over a monoid generating set $X = X_0\cup \bigcup_{\alpha\in A}X_{\alpha}$ so that:
\begin{enumerate}
\item $G_{\alpha} = \pi(X_{\alpha}^*)$ for each $\alpha\in A$.
\item If $v, w\in X^*$ are irreducible words so that $G_{\alpha}\cdot \pi(v) = G_{\alpha}\cdot \pi(w)$ for some $\alpha$, then $v = v_1u$ and $w = w_1u$ with $v_1, w_1\in X_{\alpha}^*$.
\end{enumerate}
\emph{Right relative finite complete rewriting systems} are defined similarly, but with right cosets replacing left cosets.

\begin{remark}
Note that since the empty word $\epsilon$ is an irreducible word, in the above definition we also have that if $w\in X^*$ is irreducible and $\pi(w)\in G_{\alpha}$, then $w\in X_{\alpha}^*$.
\end{remark}

We shall also need a weakening of relative finite complete rewriting systems. At first the definition will appear artificial, but when trying to build finite complete rewriting systems for graphs of groups it will prove to be rather useful.

Let $(X, \mathcal{R})$ be a finite complete rewriting system for $G$ and let $\{G_{\alpha}\}_{\alpha\in A}$ be a collection of finitely generated subgroups of $G$. We say that $(X, \mathcal{R})$ is \emph{(left) compatible with $\{G_{\alpha}\}_{\alpha\in A}$} if for each $\alpha$ there is a finite set $X_\alpha$, an epimorphism $\pi_\alpha\colon X_\alpha^*\to G_{\alpha}\leqslant G$ and a finite set of rewriting rules $\mathcal{R}_\alpha$ over the alphabet $X\cup X_{\alpha}\cup \{t_{\alpha}\}$, where $t_{\alpha}$ is a new letter, satisfying the following:
\begin{enumerate}
\item The rewriting system $(X\cup X_{\alpha}\cup \{t_{\alpha}\}, \mathcal{R}\cup \mathcal{R}_{\alpha})$ is terminating.
\item Each rule in $\mathcal{R}_\alpha$ is of the form $(t_\alpha u, vt_\alpha w)$, where $u, w\in X^*$, $v\in X^*_\alpha$ and $\pi(u) = \pi_\alpha(v)\pi(w)$.
\item If $t_\alpha u$ and $t_\alpha v$ are irreducible for $(X\cup X_\alpha\cup\{t_\alpha\}, \mathcal{R}\cup \mathcal{R}_\alpha)$ with $u, v\in X^*$, then $u = v$ if and only if $G_{\alpha}\cdot\pi(u) = G_{\alpha}\cdot \pi(v)$.
\end{enumerate}
\emph{Right compatibility} is defined similarly.

\begin{remark}
\label{rem:coset_reps}
It follows from the definition that the set of words $\mathcal{W}_{\alpha}\subseteq X^*$ so that $t_{\alpha}w$ is irreducible for the rewriting system $(X\cup X_\alpha\cup\{t_\alpha\}, \mathcal{R}\cup \mathcal{R}_\alpha)$ where $w\in \mathcal{W}_{\alpha}$, is a complete set of left $G_{\alpha}$-coset representatives.
\end{remark}

\begin{remark}
If $w_1t_{\alpha}w_2$ is a word with $w_1\in X_{\alpha}$ and $w_2\in X^*$, we see by definition that if $w_1t_{\alpha}w_2\to_{\mathcal{R}\cup \mathcal{R}_H}^*w$, then $w = w_1't_{\alpha}w_2'$ with $w_1'\in X_{\alpha}^*$, $w_2'\in X^*$ and we have that $\pi_{\alpha}(w_1)\pi(w_2) = \pi_{\alpha}(w_1')\pi(w_2')$.
\end{remark}

\begin{remark}
If $G$ admits a finite complete rewriting system relative to $\{G_{\alpha}\}_{\alpha\in A}$, then it certainly admits a finite complete rewriting system compatible with $\{G_{\alpha}\}_{\alpha\in A}$. Indeed, we may define $\mathcal{R}_{\alpha} = \{(t_{\alpha}u, ut_{\alpha}) \mid u\in X_{\alpha}\}$.
\end{remark}

Before moving to the main result of the section, we may establish some straightforward combination results for compatible rewriting systems.

\begin{lemma}
\label{lem:join}
If $(X_1, \mathcal{R}_1)$ and $(X_2, \mathcal{R}_2)$ are finite complete rewriting systems for $G_1$ and $G_2$, then $(X_1\cup X_2, \mathcal{R}_1\cup \mathcal{R}_2)$ is a finite complete rewriting system for $G_1*G_2$.

Moreover, if $(X_i, \mathcal{R}_i)$ is compatible with $\{G_{\alpha}\}_{\alpha\in A_i}$, then $(X_1\cup X_2, \mathcal{R}_1\cup \mathcal{R}_2)$ is compatible with $\{G_{\alpha}\}_{\alpha\in A_i}$.
\end{lemma}

\begin{proof}
This follows from the definitions combined with the normal form theorem for free products.
\end{proof}

\begin{proposition}
\label{prop:prod}
Let $(X, \mathcal{R})$ be a finite complete rewriting system for a group $G$. The following rewriting system:
\[
(X\cup \{z^{\pm1}\}, \mathcal{R}\cup \{(xz, zx), (xz^{-1}, z^{-1}x) \mid x\in X\})
\] 
is a finite complete rewriting system for $\Z\times G$. 

Suppose further that $(X, \mathcal{R})$ is compatible with a finite collection $\{G_{\alpha}\}_{\alpha\in A}$ of finitely generated subgroups of $G$ and let $\{k_{\alpha}\}_{\alpha\in A}$ be a collection of non-zero integers, then the above rewriting system is compatible with $\{k_{\alpha}\Z\times G_{\alpha}\}_{\alpha\in A}$.
\end{proposition}

\begin{proof}
The fact that the given finite rewriting system for $\Z\times G$ is complete is clear. It suffices to show the last statement for the case of a single subgroup $H$ and a single integer $k> 0$. Let $t$ be our new letter and let $\{y^{\pm1}\}\cup X_H$ be the generating set for $k\Z\times H$ with $\pi_{k\Z\times H}(y^{\pm1}) = z^{\pm k}$ and $\pi_{k\Z\times H}(h) = \pi_H(h)$ for all $h\in X_H$. We let
\begin{align*}
\mathcal{R}_{k\Z\times H} = &\{(tz^{k}, yt)\}\cup \{(tz^i, y^{-1}tz^{i+k}) \mid -k\leqslant i<0\}\\
						&\cup \{(tz^iw, utz^iv)\mid 0\leqslant i<k, \forall\, (t_Hw, ut_Hv)\in \mathcal{R}_H\}
\end{align*}
The rewriting system given by adding $X_H\cup\{y^{\pm1}\}$ and $\mathcal{R}_{k\Z\times H}$ satisfies the required properties.
\end{proof}

\subsection{Graphs of groups}

We refer the reader to Serre's book \cite{Serre} for the necessary details on graphs and graphs of groups. A \emph{graph} $\Gamma$ is formally a tuple $(V(\Gamma), E(\Gamma), o, \bar{ }\,)$ consisting of
\begin{enumerate}
\item The \emph{vertex set} $V = V(\Gamma)$.
\item The \emph{edge set} $E = E(\Gamma)$.
\item The \emph{origin map} $o\colon E\to V$.
\item The \emph{reverse map} $\bar{ }\,\colon E\to E$ satisfying $\overline{e}\neq e$ and $\overline{\overline{e}} = e$ for all $e\in E$.
\end{enumerate}
A \emph{spanning tree} for $\Gamma$ is a connected subgraph $T\subset \Gamma$ which contains every vertex of $\Gamma$, but which does not support any reduced cycles. A spanning tree exists for any connected graph.

A \emph{graph of groups} $\mathcal{G}$ is a tuple $(\Gamma, \{G_v\}_{v\in V}, \{G_e\}_{e\in E}, \{\partial_e\}_{e\in E})$ consisting of
\begin{enumerate}
\item A graph $\Gamma$ called the \emph{underlying graph}.
\item A group $G_v$ associated to each vertex $v\in V$ called the \emph{vertex groups}.
\item A group $G_e$ associated to each edge $e\in E$ called the \emph{edge groups}. We assume that $G_e = G_{\overline{e}}$ for each $e\in E$.
\item A monomorphism $\partial_e\colon G_e\to G_{o(e)}$.
\end{enumerate}
If $T\subset \Gamma$ is a spanning tree and, for each $v\in V$, $\langle S_v \mid R_v\rangle$ is a presentation for the vertex group $G_v$, then it is well-known that the fundamental group of $\mathcal{G}$ has presentation:
\begin{align*}
\pi_1(\mathcal{G}) = \langle \{S_v\}_{v\in V}\cup \{t_e\}_{e\in E} \mid & \{R_v\}_{v\in V}\cup \{t_e^{-1} = t_{\overline{e}}\}_{e\in E} \cup \{t_e = 1\}_{e\in E(T)} \\
										& \cup\{t_{e}^{-1}\partial_e(g)t_e = \partial_{\overline{e}}(g)\}_{e\in E, g\in G_e}\rangle. 
\end{align*}
It is also a well-known fact that different choices of spanning trees in $\Gamma$ lead to isomorphic groups.

Often one works instead with the following group:
\begin{align*}
F(\mathcal{G}, \Gamma) = \langle \{S_v\}_{v\in V}\cup \{t_e\}_{e\in E} \mid & \{R_v\}_{v\in V}\cup \{t_e^{-1} = t_{\overline{e}}\}_{e\in E}\\
										& \cup\{t_{e}^{-1}\partial_e(g)t_e = \partial_{\overline{e}}(g)\}_{e\in E, g\in G_e}\rangle. 
\end{align*}
There is a natural surjection $F(\mathcal{G}, \Gamma)\to \pi_1(\mathcal{G})$ given by quotienting out by the normal closure of $\{t_e\}_{e\in E(T)}$. We have:

\begin{lemma}
\label{lem:free_prod}
If $k = |V(\Gamma)| - 1$, we have
\[
F(\mathcal{G}, \Gamma) \cong \pi_1(\mathcal{G})*F_k.
\]
\end{lemma}

\begin{proof}
Consider the graph of groups $\mathcal{H}$ obtained from $\mathcal{G}$ by adding a vertex $v$ with trivial vertex group and, for each vertex $w\in V(\Gamma)$, adding an edge connecting $v$ with $w$. By extending the spanning tree of $\Gamma$ to a spanning tree of the underlying graph of $\mathcal{H}$ and by observing the presentations, we see that $\pi_1(\mathcal{H})\cong \pi_1(\mathcal{G})*F$ where $F$ is a free group with generating set of cardinality equal to $|V(\Gamma)| - 1$. By instead choosing the spanning tree consisting of all the edges not in $\Gamma$ and by observing the presentations, we see that $\pi_1(\mathcal{H})\cong F(\mathcal{G}, \Gamma)$. This completes the proof.
\end{proof}

\begin{theorem}[Combining finite complete rewriting systems]
\label{thm:combination}
If $\mathcal{G}$ is a finite graph of groups in which each vertex group $G_v$ admits a finite complete rewriting system left compatible with the subgroups $\{\partial_e(G_e)\}_{e\in E, o(e) = v}$, then $F(\mathcal{G}, \Gamma)$ admits a finite complete rewriting system.
\end{theorem}

\begin{proof}
If we identify all vertices of $\Gamma$ we obtain a new graph of groups $\mathcal{G}'$ with a single vertex group $G_u$ so that $G_u = *_{v\in V}G_v$. We have $\pi_1(\mathcal{G}') \cong F(\mathcal{G}, \Gamma)$. Moreover, $\mathcal{G}'$ satisfies all the assumptions of the theorem by \cref{lem:join}. Thus, we may assume without loss that $\Gamma$ consists of a single vertex, $u$, for simplicity.

Let $(X_u, \mathcal{R}_u)$ be the finite complete rewriting system for $G_u$ that is compatile with $\{\partial_e(G_e)\}_{e\in E}$ which exists by assumption. For each $e\in E$, denote by $\pi_e\colon X_{e}^*\to G_e$ the surjection, by $\mathcal{R}'_e$ the rules and by $t_e'$ the new letter given by the definition of compatible rewriting systems. For each $x\in X_e$, choose some word $w_x$ over $X_{u}$ so that 
\[
\pi(w_x) = \partial_{\overline{e}}(\pi_e(x)) = \pi(t_{e}^{-1}xt_e).
\]
For each $e\in E(\Gamma)$, consider the rules:
\[
\mathcal{R}_e = \{(t_{\overline{e}}u, w_xt_{\overline{e}}v) \mid \,\forall\, (t_{e}'u, xt_{e}'v)\in \mathcal{R}_e'\}.
\]
Consider the sets
\begin{align*}
X &= X_u \cup \{t_e, t_e^{-1}\}_{e\in E}\\
\mathcal{R} &= \mathcal{R}_u\cup \bigcup_{e\in E}\mathcal{R}_e\cup \{(t_et_{\overline{e}}, \epsilon), (t_{\overline{e}}t_e, \epsilon)\}_{e\in E}\cup \{(t_{e}^{-1}, t_{\overline{e}})\}_{e\in E}
\end{align*}
and the rewriting system $(X, \mathcal{R})$. It is clear that this is a finite monoid presentation of $F(\mathcal{G}, \Gamma)$. We just need to show that it is complete.

If $\mathcal{Q}$ is a set of rewriting rules over an alphabet $Y$, define $c_{\mathcal{Q}}(x)$ to be the maximal integer $n$ so that $x = y_0\to_{\mathcal{Q}}y_1\to_{\mathcal{Q}}\ldots\to_{\mathcal{Q}}y_n$. This is well-defined whenever $(Y, \mathcal{Q})$ is finite and terminating (this follows from an application of K\"onig's Lemma). 

Since $(X_u\cup \{t_{e}\}, \mathcal{R}_u\cup\mathcal{R}_{\overline{e}})$ is a finite terminating rewriting system by assumption, we see that $c_{\mathcal{R}_u\cup\mathcal{R}_{\overline{e}}}(x)$ is well-defined for any word $x\in (X_u\cup\{t_e\})^*$. If $x\in X^*$ is any word, there is a unique factorisation $x = x_0t_{e_1}x_1\ldots t_{e_k}x_k$ so that for each $i\geqslant 0$ we have $x_i\in X_u^*$. We may associate to $x$ a tuple of integers $C(x) = (n_0, \ldots, n_k)$ where $n_i = c_{\mathcal{R}_u\cup\mathcal{R}_{\overline{e}}}(t_{e_i}x_i)$ for each $i\geqslant 1$ and $n_0 = c_{\mathcal{R}_u}(x_0)$. We write $(n_0, \ldots, n_k)>(m_0, \ldots, m_q)$ if $k>q$ or if $k = q$ and $n_i>m_i$ for some $i$ and $n_j = m_j$ for all $j>i$. We see that for each $x, y\in X^*$ so that $x\to_{\mathcal{R}}y$, we have $C(x)>C(y)$. Thus, $(X, \mathcal{R})$ is terminating.

Now let $x, y\in X^*$ be two irreducibles. We wish to show that $\pi(x) = \pi(y)$ if and only if $x = y$. Let $x = x_0t_{e_1}x_1\ldots t_{e_n}x_n$ be the unique factorisation as above. We have:
\begin{enumerate}
\item For each $i\geqslant 0$, we have that $t_{e_i}x_i$ is an irreducible for the rewriting system $(X_u\cup\{t_{e_i}\}, \mathcal{R}_u\cup\mathcal{R}_{\overline{e}_i})$.
\item For each $i\geqslant 1$, we have $\overline{e}_i =  e_{i+1}$ only if $x_{i}\neq \epsilon$.
\end{enumerate}
Similarly, there is a unique factorisation $y = y_0t_{f_1}y_1\ldots t_{f_m}y_m$ with the same properties. Now by \cref{rem:coset_reps} and by the normal form theorem for graphs of groups \cite{Higgins1976}, the above implies that $x = y$ as claimed.
\end{proof}

\begin{remark}
It is perhaps worth remarking here that \cref{thm:combination} cannot be improved too much since there are groups with undecidable word problem that can be constructed from free groups by finitely many HNN-extensions over finitely generated groups \cite{CM99}.
\end{remark}

By combining \cref{lem:free_prod} with \cref{thm:combination}, we obtain the following corollary.

\begin{corollary}
\label{cor:combination}
If $\mathcal{G}$ is a finite graph of groups in which each vertex group $G_v$ admits a finite complete rewriting system (left) compatible with the subgroups $\{\partial_e(G_e)\}_{e\in E, o(e) = v}$, then the group $\pi_1(\mathcal{G})*F$ admits a finite complete rewriting system where $F$ is a free group with $\rk(F) = |V| -1$.
\end{corollary}

We may slightly improve \cref{cor:combination}.

\begin{proposition}
\label{prop:producting}
If $\mathcal{G}$ is a finite graph of groups in which each vertex group $G_v$ admits a finite complete rewriting system (left) compatible with the subgroups $\{\partial_e(G_e)\}_{e\in E, o(e) = v}$, then the group $F*(*_{i=1}^n\pi_1(\mathcal{G}))$ admits a finite complete rewriting system for all $n\geqslant 1$, where $F$ is a free group with $\rk(F) = |V| - 1$ .
\end{proposition}

\begin{proof}
Choose a spanning tree $T\subset \Gamma$ and define a new graph of groups $\mathcal{G}_n$ obtained by taking $n$ copies of $\mathcal{G}$ and identifying each copy of $T$ together. The new vertex groups are simply free products of $n$ copies of the old vertex groups. By \cref{lem:join} this new graph of groups satisfies the hypotheses of \cref{thm:combination}. We have that $\pi_1(\mathcal{G}_n) \cong *_{i=1}^n\pi_1(\mathcal{G})$. Now the result holds by \cref{cor:combination}.
\end{proof}

\begin{corollary}
\label{cor:combination2}
If $\mathcal{G}$ is a finite graph of groups in which each vertex group $G_v$ admits a finite complete rewriting system (left) compatible with the subgroups $\{\partial_e(G_e)\}_{e\in E, o(e) = v}$, then the group $\pi_1(\mathcal{G})*F_2$ admits a finite complete rewriting system.
\end{corollary}

\begin{proof}
By \cref{prop:producting}, $F_{k+1}*(*_{i=1}^k\pi_1(\mathcal{G}))$ admits a finite complete rewriting system where $k = |V|$. Consider any epimorphism $\phi\colon \pi_1(\mathcal{G})*F_2\to \Z/k\Z$ factoring through the projection to $F_2$. Then $\ker(\phi) \cong F_{k+1}*(*_{i=1}^k\pi_1(\mathcal{G}))$. Now by \cref{prop:finite_index}, $\pi_1(\mathcal{G})*F_2$ admits a finite complete rewriting system.
\end{proof}

\section{Finite complete rewriting systems for graphs of free groups}

The aim of this section is to prove the following theorem.

\MainTheoremA*

\subsection{Compatible rewriting systems for free groups}

The key missing ingredient for the proof of \cref{thm:main_free} is the following.

\begin{theorem}
\label{thm:compatible_free}
Let $F = F(X)$ be the free group freely generated by the finite set $X$ and let $\pi\colon (X^{\pm1})^*\to F$ be the natural monoid epimorphism. If $\mathcal{R} = \{(xx^{-1}, \epsilon), (x^{-1}x, \epsilon)\, \forall x\in X\}$, then the rewriting system $(X^{\pm1}, \mathcal{R})$ is a finite complete rewriting system for $F$.

Moreover, if $\{F_{\alpha}\}_{\alpha\in A}$ is a finite collection of finitely generated subgroups, then the rewriting system $(X^{\pm1}, \mathcal{R})$ is (left and right) compatible with $\{F_{\alpha}\}_{\alpha\in A}$.
\end{theorem}

\begin{proof}
The fact that the given rewriting system is finite and complete is clear. Moreover, for the compatibility statement, it suffices to prove it for a single group, say $H= F_{\alpha}$. We will prove left compatibility as right compatibility is similar.

Let $\Gamma$ denote the Schreier coset graph for $H$. Recall that the vertices of $\Gamma$ are left cosets $Hh$, oriented edges are labelled by elements in the generating set $X^{\pm1}$ and two vertices $Hh_1, Hh_2$ are connected by an oriented edge labelled by $x$ if and only if $Hh_1\pi(x) = Hh_2$. Let $T\subset \Gamma$ be a spanning tree for $\Gamma$. Since $H$ is finitely generated, the number of oriented edges $E$ in $\Gamma - T$ is finite (see \cite{St83}). For each oriented edge $e\in E$, let $p_e\in (X^{\pm1})^*$ be the word labelling the unique geodesic path in $T$ connecting $H$ with the origin of $e$, let $x_e$ be the label of $e$ and let $q_e\in X^*$ be the word labelling the unique geodesic path in $T$ connecting the target of $e$ with $H$. By definition we have $\pi(p_ex_eq_e)\in H$. Denote by $X_H = \{w_e \mid e\in E\}$, and define $\pi_H\colon X_H\to H$ by $\pi_H(w_e) = p_ex_eq_e$. Note that $\pi_H$ is surjective. Now consider the following rules
\[
\mathcal{R}_H = \{(t_Hp_ex_e, w_et_Hq_e^{-1}) \mid e\in E\}.
\]
We claim that $\mathcal{R}_H$ is the required set of rewriting rules. 

We begin by showing that $(X^{\pm1}\cup X_H\cup\{t_H\}, \mathcal{R}\cup\mathcal{R}_H)$ is terminating. Let $w$ be any word. We may write $w = w_0u_1w_1\ldots u_kw_k$ where $w_i\in (X^{\pm1})^*$ and $u_i\in (X_H\cup \{t_H\})^*$. Now each word $w_i$ determines a unique path in $\Gamma$ connecting $H$ with another vertex. Denote by $\#_E(w_i)$ the number of times the corresponding path traverses an edge in $E$. Denote by $\#_E(w) = \sum_{i=0}^k\#_E(w_i)$. Then each application of a rule in $\mathcal{R}_H$ strictly decreases $\#_E(w)$ and each application of a rule in $\mathcal{R}$ strictly decreases the length of $w$ and does not increase $\#_E(w)$. Thus, the rewriting system is certainly terminating.

The second required property follows by definition. 

Now we establish the third property. Let $t_{H}u, t_{H}v$ be two irreducibles for $(X^{\pm1}\cup X_{H}\cup \{t_H\}, \mathcal{R}\cup \mathcal{R}_H)$ with $u, v\in (X^{\pm1})^*$. Since $u, v$ are irreducible, they are freely reduced. In particular, the unique paths in $\Gamma$ leading out from $H$ with labels $u, v$ respectively do not have any backtracking. Thus, if either of the paths traverse an edge in $E$, we must have that $p_ex_e$ is a prefix of $u$ or $v$ (where $e$ is the first edge in $E$ traversed), contradicting the assumption that $u, v$ were irreducible. Thus, neither of the two paths traverse any edge in $E$ and so are supported in the spanning tree $T$. By definition of the Schreier graph, $H\pi(u) = H\pi(v)$ if and only if the corresponding paths end at the same vertex. But since the paths are supported in a tree, they are geodesic and they end at the same vertex precisely if they coincide. Thus $H\pi(u) = H\pi(v)$ if and only if $u = v$. This establishes the third required property.
\end{proof}

\begin{remark}
\label{rem:algorithmic}
Via Stallings' graphs and Stallings foldings \cite{St83} (see also \cite{KM02}), the rewriting rules from the proof of \cref{thm:compatible_free} can be given algorithmically.
\end{remark}

\subsection{Proof of \cref{thm:main_free}}

The first item from \cref{thm:main_free} is the combination of \cref{thm:compatible_free} with \cref{cor:combination,cor:combination2}. For the second item, note that if $G$ admits a proper finite index subgroup $H$, then $H$ also splits as a finite graph of finitely generated free groups. Thus, by \cref{thm:compatible_free} and \cref{cor:combination2}, $H*F_2$ admits a finite complete rewriting system and by \cref{prop:finite_index}, $G*\Z$ admits a finite complete rewriting system as claimed.

\subsection{Applications}

In this section we discuss applications of \cref{thm:main_free}.

\MainTheoremB*

\begin{proof}
Let $H\leqslant G$ be a finitely generated subgroup. By work of Feighn--Handel \cite[Proposition 2.3]{Feighn1999}, there is a finitely generated free group $A*B = F\leqslant \F$ so that $\phi = \psi\mid_{A}\colon A\to \psi(A)$ is an isomorphism and $H\cong F*_{\phi}$. In particular, $H$ splits as a HNN-extension of a finitely generated free group and so admits a finite complete rewriting system by \cref{cor:HNN}. The last sentence in the statement of the result then follows by applying a result from \cite{Groves1993} which says that the property of having a finite complete rewriting system is preserved when taking finite index extensions of groups.
\end{proof}

In the introduction we stated some applications of \cref{rewriting_main}. In this subsection we discuss those applications and give proofs for them.

We also have the following effective version of this result.

\begin{theorem}
\label{thm:computingFCRSGeneral} 
There is an algorithm that takes any finite presentation of a group that is virtually a subgroup of a mapping torus of a free group as input and outputs a finite complete presentation defining that group.
\end{theorem}

\begin{proof}
Let $G$ be a group given by a finite presentation such that $G$ is virtually a subgroup of a mapping torus of a free group. We can enumerate presentations of finite index subgroups of $G$ and perform Tietze transformations on the presentations until we have a presentation for a finite index subgroup $H$ of $G$ of the form given by the HNN-extension from the statement of \cite[Proposition 2.3]{Feighn1999}. Then, by combining the compatible rewriting systems from \cref{thm:compatible_free} (see \cref{rem:algorithmic}) with the rewriting systems in the proof of \cref{thm:combination}, we may write down a finite complete presentation for $H$. Then we may use the argument in \cite{Groves1993} to transform the finite complete rewriting system for the finite index subgroup into one for $G$. 
\end{proof}

We saw in the introduction in Theorem~\ref{thm:locallyQuasConv} that this can be applied to locally quasi-convex hyperbolic and virtually compact special groups. We provide a proof of that result here. 

\begin{proof}[Proof of Theorem~\ref{thm:locallyQuasConv}] 
Let $G$ be a locally quasi-convex hyperbolic and virtually compact special group. It follows from \cite[Theorem A.2]{fisher2026coherent} that $G$ is virtually free-by-cyclic. Hence by  \cref{rewriting_main} 
the group $G$ and all of its finitely generated subgroups admit complete rewriting systems.  
For the second clause in the statement of Theorem~\ref{thm:locallyQuasConv}, it follows from work of Grunschlag \cite[Proposition 6.1]{Grunschlag1999} that locally quasi-convex hyperbolic groups are effectively coherent. In other words, there is an algorithm that, given a finite presentation for $G$ and a finite subset $S\subset G$ computes a finite presentation for the subgroup $H$ of $G$ generated by $S$. But since $G$ is virtually free-by-cyclic it follows that $H$ also is, and the result follows by applying Theorem~\ref{thm:computingFCRSGeneral}. 
\end{proof}

The main result of this section may also be applied to subgroups of surface-by-cyclic groups. A surface-by-cyclic group is a group $S \rtimes_{\psi} \Z$, where $S$ is the fundamental group of a connected compact surface. 

\begin{theorem}\label{thm:SurfaceByCyclic} 
Any finitely generated subgroup of a surface-by-cyclic group admits a finite complete rewriting system. 
  \end{theorem}
\begin{proof} 
Let $H$ be a finitely generated subgroup of a surface-by-cyclic group $G = S \rtimes_{\psi} \mathbb{Z}$.

If $H$ has finite index in $G$ then $H$ is also a surface-by-cyclic group so $H = T \rtimes_{\phi} \mathbb{Z}$ for some surface group $T$. Since surface groups have finite complete rewriting systems \cite{LeChenadec1986} it follows that $T$ admits a finite complete rewriting systems (since they are also free-by-cyclic, we could also use \cref{rewriting_main}). Then applying \cite[Proof of part (1) of  Theorem (2)]{Groves1993} it follows that $H$ had a finite complete rewriting system. 

If $H$ has infinite index in $G$ then, since surface groups are free-by-cyclic and since every infinite index subgroup of a surface group is a free group, it follows that $H$ is a finitely generated free-by-cyclic group. Hence $H$ has a finite complete rewriting system by Theorem~\ref{rewriting_main}. \end{proof}

\section{Finite complete rewriting systems for three-manifold groups}
\label{sec:3man}

In this section we prove the following, almost answering a question which appears at the end of \cite{Hermiller1999}.

\MainTheoremC*

We will need to understand the closed graph manifold case separately via \cref{thm:combination}. We will then use \cref{rewriting_main} along with general three-manifold results to conclude the proof. See \cite{AFW15} for background.

\subsection{The irreducible but not closed graph manifold case}

We begin by handling the simpler case in which $M^3$ does not contain any closed graph manifold pieces.

\begin{proposition}
\label{prop:irreducible}
Let $M^3$ be a compact irreducible three-manifold with (possibly empty) incompressible boundary. If $M^3$ is not a closed graph manifold that does not virtually fibre, then $\pi_1(M^3)$ admits a finite complete rewriting system.
\end{proposition}

\begin{proof}
By the sphere theorem (see \cite[Theorem 1.3.2]{AFW15}), since $M^3$ is irreducible we have that $\pi_2(M^3) = 0$. 

Let us first suppose that $M^3$ has non-empty boundary. Every $n$-manifold with non-empty boundary deformation retracts onto a spine, which is a CW-complex of dimension $n-1$. Since $\pi_2(M^3) = 0$, $M^3$ deformation retracts onto an aspherical 2-complex. In particular, $\pi_1(M^3)$ has cohomological dimension at most two and so is virtually free-by-cyclic by \cite{KL24b}. Now by \cref{rewriting_main} combined with \cref{prop:finite_index}, $\pi_1(M^3)$ admits a finite complete rewriting system.

Now suppose that $M^3$ is closed. The geometric case is handled by Hermiller--Shapiro \cite{Hermiller1999}. Since $M^3$ is not a closed graph manifold that does not virtually fibre, it follows from \cite{WY97,Ag13,PW18} that $M^3$ does virtually fibre. Thus, $G$ has a finite index subgroup $H \cong S\rtimes\Z$ where $S$ is the fundamental group of a closed surface. Hence, by Theorem~\ref{thm:SurfaceByCyclic} combined with \cref{prop:finite_index}, $\pi_1(M^3)$ has a finite complete rewriting system.
\end{proof}

\subsection{The closed graph manifold case}

A useful lemma, which will be our basis for handling the closed graph manifold case, was proven by Wang--Yu \cite[Lemma 0.4]{WY97}.

\begin{lemma}
\label{lem:product}
Every compact graph manifold has a finite sheeted cover with JSJ-decomposition in which each geometric piece $M_i$ is homeomorphic to $S_i\times S^1$, where $S_i$ is a compact orientable surface. 
\end{lemma}

The idea will now be to apply \cref{thm:combination} to the decomposition provided by \cref{lem:product}. In order to do this we need to concoct finite complete rewriting systems for $\pi_1(S\times S^1)$ for each orientable surface $S$ with boundary, compatible with the fundamental groups of the components of $\partial S\times S^1$.

\begin{proposition}
\label{prop:seifert}
If $S$ is a compact surface with boundary, then $G = \pi_1(S\times S^1)$ admits a finite complete rewriting system compatible with $\{\pi_1(B\times S^1)\}_{B\subset \partial S}$.
\end{proposition}

\begin{proof}
Since $S$ has boundary, $\pi_1(S)$ is free. By \cref{thm:main_free}, $\pi_1(S)$ admits a finite complete rewriting system $\mathcal{R}$ (over a generating set $X$) compatible with $\{\pi_1(B)\}_{B\subset \partial S}$. Combined with \cref{prop:prod}, we see that $G$ admits a finite complete rewriting system compatible with $\{\pi_1(B\times S^1)\}_{B\subset \partial S}$.
\end{proof}

\begin{theorem}
\label{thm:graph_manifold}
If $M_1, \ldots, M_n$ are closed graph manifolds, then $(*_{i=1}^n\pi_1(M_i))*\Z$ admits a finite complete rewriting system.
\end{theorem}

\begin{proof}
By \cref{lem:product} combined with \cref{prop:seifert} and \cref{lem:join}, $G = *_{i=1}^n\pi_1(M_i)$ has a finite index subgroup $H$ isomorphic to a finite graph of groups satisfying the hypotheses of \cref{thm:combination}. After passing to another finite index subgroup, we may assume that $H$ is proper (note that $G$ is residually finite). Thus, by \cref{cor:combination2}, $H*F_2$ admits a finite complete rewriting system. Using the fact that $[G:H]\geqslant 2$, we may use \cref{prop:finite_index} to complete the proof.
\end{proof}

\subsection{The general case}

Before proving \cref{thm:3-manifold:intro}, we prove a conditional statement without requiring the free factor.

\begin{proposition}
\label{prop:conditional}
Suppose that $M^3$ is a three-manifold with $G = \pi_1(M^3)$ finitely generated and so that every finitely generated subgroup $H\leqslant \pi_1(M^3)$ isomorphic to the fundamental group of a closed graph manifold that does not virtually fibre admits a finite complete rewriting system. Then $\pi_1(M^3)$ admits a finite complete rewriting system.
\end{proposition}

\begin{proof}
By \cite[Proposition 2.2]{KL24b}, $G$ has a finite index subgroup isomorphic to $F*(*_{i=1}^n\pi_1(M_i))$, where $F$ is a finitely generated free group and where each $M_i$ is a compact irreducible three-manifold with (possibly empty) incompressible boundary. Then combining \cref{lem:join}, \cref{prop:finite_index} and \cref{prop:irreducible} completes the proof.
\end{proof}

\begin{proof}[Proof of \cref{thm:3-manifold:intro}]
By \cite[Proposition 2.2]{KL24b}, $G = \pi_1(M)$ has a finite index subgroup isomorphic to $F*(*_{i=1}^n\pi_1(M_i))$, where $F$ is a finitely generated free group and where each $M_i$ is a compact irreducible three-manifold with (possibly empty) incompressible boundary. Then combining \cref{lem:join}, \cref{prop:finite_index}, \cref{prop:irreducible} and \cref{thm:graph_manifold} completes the proof.
\end{proof}

We may now also prove \cref{thm:autostackable:intro} from the introduction.

\begin{proof}[Proof of \cref{thm:autostackable:intro}]
By \cite[Theorem 3.4]{Brittenham2016}, being autostackable passes to finite index overgroups. By \cite[Theorem 3.2]{Brittenham2016}, free products of autostackable groups are autostackable. By \cite[Theorem 6.2]{Hermiller2018}, closed orientable irreducible three-manifold groups are autostackable. Since groups with finite complete rewriting systems are autostackable by \cite[Corollary 5.4]{Brittenham2014}, we may use \cref{prop:irreducible} combined with the above to complete the proof in the same way as that of \cref{thm:3-manifold:intro}. 
\end{proof}

\section{Concluding remarks and open questions} 

We end the paper with some remarks on open questions. One obvious lingering question is the following:

\begin{question}
Does \cref{thm:combination} hold if $F(\mathcal{G}, \Gamma)$ is replaced with $\pi_1(\mathcal{G})$?
\end{question}

It remains open whether all one-relator groups admit finite complete rewriting systems (although they do not, in general, admit finite complete rewriting systems over the given generating set \cite{jantzen1985note}). It is open even for the Baumslag--Gersten group \cite{hermiller2018hnn}. We also do not know if all one-relator groups are autostackable. If either of these results were true it would give a rational set of normal forms. This prompts us to ask:   

\begin{question}
Does every one-relator group have a rational cross-section?     
\end{question}

We do not know whether the $\Z$ factor in \cref{thm:3-manifold:intro}, or in \cref{cor:FreeProdAmalg}, can be dropped. More generally, we ask:

\begin{question}
If $G*\Z$ admits a finite complete rewriting system, does $G$ admit a finite complete rewriting system?
\end{question}

This is a simpler instance of the general question for free products, asked by Otto (see \cite{Pride_Wang_2000}).


\begin{thebibliography}{10}

\bibitem{AdamsGrobner}
W.~W. Adams and P.~Loustaunau.
\newblock {\em An introduction to {G}r\"obner bases}, volume~3 of {\em Graduate
  Studies in Mathematics}.
\newblock American Mathematical Society, Providence, RI, 1994.

\bibitem{Ag13}
I.~Agol.
\newblock The virtual {H}aken conjecture.
\newblock {\em Doc. Math.}, 18:1045--1087, 2013.
\newblock With an appendix by Agol, Daniel Groves, and Jason Manning.

\bibitem{AGS}
D.~J. Anick.
\newblock On the homology of associative algebras.
\newblock {\em Trans. Amer. Math. Soc.}, 296(2):641--659, 1986.

\bibitem{AFW15}
M.~Aschenbrenner, S.~Friedl, and H.~Wilton.
\newblock {\em 3-manifold groups}.
\newblock EMS Series of Lectures in Mathematics. European Mathematical Society
  (EMS), Z\"{u}rich, 2015.

\bibitem{behrstock2008quasi}
J.~A. Behrstock and W.~D. Neumann.
\newblock Quasi-isometric classification of graph manifold groups.
\newblock {\em Duke Math. J.}, 141(2):217--240, 2008.

\bibitem{BF91}
M.~Bestvina and M.~Feighn.
\newblock Bounding the complexity of simplicial group actions on trees.
\newblock {\em Inventiones mathematicae}, 103(3):449--469, 1991.

\bibitem{Bodart2024}
C.~Bodart.
\newblock Rational cross-sections, bounded generation, and orders on groups.
\newblock {\em J. Lond. Math. Soc. (2)}, 109(6):e12920, 2024.

\bibitem{bridson1993combings}
M.~R. Bridson.
\newblock Combings of semidirect products and 3-manifold groups.
\newblock {\em Geom. Funct. Anal.}, 3:263--278, 1993.

\bibitem{BG1996}
M.~R. Bridson and R.~H. Gilman.
\newblock Formal language theory and the geometry of 3-manifolds.
\newblock {\em Comment. Math. Helv.}, 71(1):525--555, 1996.

\bibitem{brittenham2015uniform}
M.~Brittenham and S.~Hermiller.
\newblock A uniform model for almost convexity and rewriting systems.
\newblock {\em J. Group Theory}, 18(5):805--828, 2015.

\bibitem{Brittenham2014}
M.~Brittenham, S.~Hermiller, and D.~Holt.
\newblock Algorithms and topology of {C}ayley graphs for groups.
\newblock {\em J. Algebra}, 415:112--136, 2014.

\bibitem{Brittenham2016}
M.~Brittenham, S.~Hermiller, and A.~Johnson.
\newblock Homology and closure properties of autostackable groups.
\newblock {\em J. Algebra}, 452:596--617, 2016.

\bibitem{Hermiller2018}
M.~Brittenham, S.~Hermiller, and T.~Susse.
\newblock Geometry of the word problem for 3-manifold groups.
\newblock {\em J. Algebra}, 499:111--150, 2018.

\bibitem{Chouraqui2009}
F.~Chouraqui.
\newblock Rewriting systems in alternating knot groups with the {D}ehn
  presentation.
\newblock {\em Geom. Dedicata}, 138:173--192, 2009.

\bibitem{CM99}
D.~J. Collins and C.~F. Miller.
\newblock The word problem in groups of cohomological dimension 2.
\newblock In {\em Groups St. Andrews 1997 in Bath. Selected papers of the
  international conference, Bath, UK, July 26--August 9, 1997. Vol. 1}, pages
  211--218. Cambridge: Cambridge University Press, 1999.

\bibitem{Dekov1997}
D.~V. Dekov.
\newblock Finite complete rewriting systems for groups.
\newblock {\em Comm. Algebra}, 25(12):4023--4028, 1997.

\bibitem{Diekert2010}
V.~Diekert, A.~J. Duncan, and A.~G. Myasnikov.
\newblock Geodesic rewriting systems and pregroups.
\newblock In {\em Combinatorial and geometric group theory}, Trends Math.,
  pages 55--91. Birkh\"{a}user/Springer Basel AG, Basel, 2010.

\bibitem{DrutuSapir2005}
C.~Drutu and M.~Sapir.
\newblock Non-linear residually finite groups.
\newblock {\em Journal of Algebra}, 284(1):174--178, 2005.

\bibitem{Feighn1999}
M.~Feighn and M.~Handel.
\newblock Mapping tori of free group automorphisms are coherent.
\newblock {\em Ann. of Math. (2)}, 149(3):1061--1077, 1999.

\bibitem{fisher2026cohomological}
S.~P. Fisher.
\newblock On the cohomological dimension of kernels of maps to $\mathbb{Z}$.
\newblock {\em Geom. Topol.}, 30(1):373--388, 2026.

\bibitem{fisher2026coherent}
S.~P. Fisher, M.~Linton, and P.~S{\'a}nchez-Peralta.
\newblock Coherent {RFRS} groups.
\newblock {\em arXiv preprint arXiv:2603.16763}, 2026.

\bibitem{Gersten}
S.~M. Gersten.
\newblock Subgroups of word hyperbolic groups in dimension 2.
\newblock {\em J. London Math. Soc. (2)}, 54(2):261--283, 1996.

\bibitem{Grigorchu1997}
R.~Grigorchuk and T.~Nagnibeda.
\newblock Complete growth functions of hyperbolic groups.
\newblock {\em Invent. Math.}, 130(1):159--188, 1997.

\bibitem{Groves1993}
J.~R.~J. Groves and G.~C. Smith.
\newblock Soluble groups with a finite rewriting system.
\newblock {\em Proc. Edinburgh Math. Soc. (2)}, 36(2):283--288, 1993.

\bibitem{Grunschlag1999}
Z.~Grunschlag.
\newblock Computing angles in hyperbolic groups.
\newblock In {\em Groups, languages and geometry ({S}outh {H}adley, {MA},
  1998)}, volume 250 of {\em Contemp. Math.}, pages 59--88. Amer. Math. Soc.,
  Providence, RI, 1999.

\bibitem{HW2010}
M.~F. Hagen and D.~T. Wise.
\newblock Special groups with an elementary hierarchy are virtually
  free-by-$\mathbb{Z}$.
\newblock {\em Groups Geom. Dyn.}, 4(3):597--603, 2010.

\bibitem{hermiller2018hnn}
S.~Hermiller and C.~Mart{\'\i}nez-P{\'e}rez.
\newblock {HNN} extensions and stackable groups.
\newblock {\em Groups, Geometry, and Dynamics}, 12(3):1123--1158, 2018.

\bibitem{Hermiller1995}
S.~Hermiller and J.~Meier.
\newblock Algorithms and geometry for graph products of groups.
\newblock {\em J. Algebra}, 171(1):230--257, 1995.

\bibitem{hermiller2001measuring}
S.~Hermiller and J.~Meier.
\newblock Measuring the tameness of almost convex groups.
\newblock {\em Trans. Amer. Math. Soc.}, 353(3):943--962, 2001.

\bibitem{Hermiller1999}
S.~Hermiller and M.~Shapiro.
\newblock Rewriting systems and geometric three-manifolds.
\newblock {\em Geom. Dedicata}, 76(2):211--228, 1999.

\bibitem{hermiller1994rewriting}
S.~M. Hermiller.
\newblock Rewriting systems for {C}oxeter groups.
\newblock {\em J. Pure Appl. Algebra}, 92(2):137--148, 1994.

\bibitem{Hermiller1999Second}
S.~M. Hermiller and J.~Meier.
\newblock Artin groups, rewriting systems and three-manifolds.
\newblock {\em J. Pure Appl. Algebra}, 136(2):141--156, 1999.

\bibitem{Higgins1976}
P.~J. Higgins.
\newblock The fundamental groupoid of a graph of groups.
\newblock {\em J. London Math. Soc. (2)}, 13(1):145--149, 1976.

\bibitem{HoltBook}
D.~F. Holt, B.~Eick, and E.~A. O'Brien.
\newblock {\em Handbook of computational group theory}.
\newblock Discrete Mathematics and its Applications (Boca Raton). Chapman \&
  Hall/CRC, Boca Raton, FL, 2005.

\bibitem{JZL2025}
A.~Jaikin-Zapirain and M.~Linton.
\newblock On the coherence of one-relator groups and their group algebras.
\newblock {\em Ann. of Math. (2)}, 201(3):909--959, 2025.

\bibitem{jantzen1985note}
M.~Jantzen.
\newblock A note on a special one-rule semi-thue system.
\newblock {\em Inform. Process. Lett.}, 21(3):135--140, 1985.

\bibitem{KM02}
I.~Kapovich and A.~Myasnikov.
\newblock Stallings foldings and subgroups of free groups.
\newblock {\em J. Algebra}, 248(2):608--668, 2002.

\bibitem{KL24b}
D.~Kielak and M.~Linton.
\newblock Virtually free-by-cyclic groups.
\newblock {\em Geom. Funct. Anal.}, 34(5):1580--1608, 2024.

\bibitem{LeChenadec1986}
P.~Le~Chenadec.
\newblock A catalogue of complete group presentations.
\newblock {\em J. Symbolic Comput.}, 2(4):363--381, 1986.

\bibitem{linton2025geometry}
M.~Linton.
\newblock The geometry of subgroups of mapping tori of free groups.
\newblock {\em arXiv preprint arXiv:2510.03145}, 2025.

\bibitem{Li26}
M.~Linton.
\newblock Embedding finitely generated free-by-cyclic groups in
  {{\(\{\)}}finitely generated free{{\(\}\)}}-by-cyclic groups.
\newblock {\em Int. Math. Res. Not.}, 2026(4):15, 2026.
\newblock Id/No rnag020.

\bibitem{linton2025theory}
M.~Linton and C.-F. Nyberg-Brodda.
\newblock The theory of one-relator groups: history and recent progress.
\newblock {\em arXiv preprint arXiv:2501.18306}, 2025.

\bibitem{LYQ97}
J.~Luecke and Y.-Q. Wu.
\newblock Relative {E}uler number and finite covers of graph manifolds.
\newblock In {\em Geometric topology ({A}thens, {GA}, 1993)}, volume 2.1 of
  {\em AMS/IP Stud. Adv. Math.}, pages 80--103. Amer. Math. Soc., Providence,
  RI, 1997.

\bibitem{Otto1997}
F.~Otto and Y.~Kobayashi.
\newblock Properties of monoids that are presented by finite convergent
  string-rewriting systems---a survey.
\newblock In {\em Advances in algorithms, languages, and complexity}, pages
  225--266. Kluwer Acad. Publ., Dordrecht, 1997.

\bibitem{Otto1984}
F.~Otto.
\newblock Finite complete rewriting systems for the {J}antzen monoid and the
  {G}reendlinger group.
\newblock {\em Theoret. Comput. Sci.}, 32(3):249--260, 1984.

\bibitem{papasoglu2002quasi}
P.~Papasoglu and K.~Whyte.
\newblock Quasi-isometries between groups with infinitely many ends.
\newblock {\em Commentarii Mathematici Helvetici}, 77(1):133--144, 2002.

\bibitem{Pride_Wang_2000}
S.~J. Pride and J.~Wang.
\newblock Subgroups of finite index in groups with finite complete rewriting
  systems.
\newblock {\em Proc. Edinburgh Math. Soc. (2)}, 43(1):177–183, 2000.

\bibitem{PW18}
P.~Przytycki and D.~T. Wise.
\newblock Mixed 3-manifolds are virtually special.
\newblock {\em J. Amer. Math. Soc.}, 31(2):319--347, 2018.

\bibitem{Serre}
J.-P. Serre.
\newblock {\em Trees}.
\newblock Springer Monographs in Mathematics. Springer-Verlag, Berlin, 2003.
\newblock Translated from the French original by John Stillwell, Corrected 2nd
  printing of the 1980 English translation.

\bibitem{St83}
J.~R. Stallings.
\newblock Topology of finite graphs.
\newblock {\em Invent. Math.}, 71(3):551--565, 1983.

\bibitem{WY97}
S.~Wang and F.~Yu.
\newblock Graph manifolds with non-empty boundary are covered by surface
  bundles.
\newblock {\em Math. Proc. Cambridge Philos. Soc.}, 122(3):447--455, 1997.

\bibitem{WiseWhatsNext}
D.~T. Wise.
\newblock {\em An Invitation to Coherent Groups}, pages 326--414.
\newblock Princeton University Press, Princeton, 2020.

\end{thebibliography}

\end{document}